\documentclass{amsart}

\usepackage[T1]{fontenc}
\usepackage[latin9]{inputenc}
\usepackage{paralist}
\usepackage[colorlinks=true]{hyperref}
\hypersetup{urlcolor=blue, citecolor=red}
\usepackage{textcomp}
\usepackage{mathtools}
\usepackage{amssymb}
\usepackage{graphicx}
\usepackage{esint}
\usepackage{tikz}
\usetikzlibrary{arrows}
\usepackage{enumitem}
\usepackage{dsfont}
\usepackage{bbm}

\newtheorem{thm}{Theorem}[section]
\newtheorem{cor}[thm]{Corollary}
\newtheorem{lem}[thm]{Lemma}
\newtheorem{prop}[thm]{Proposition}
\theoremstyle{definition}
\newtheorem{defn}[thm]{Definition}
\theoremstyle{remark}
\newtheorem{rem}[thm]{Remark}

\numberwithin{equation}{section}
\numberwithin{figure}{section}

\newcommand{\defeq}{\coloneqq}
\newcommand{\eqdef}{\eqqcolon}
\renewcommand{\Re}{\mathfrak{Re}}
\newcommand{\vertiii}[1]{{\left\vert\kern-0.25ex\left\vert\kern-0.25ex\left\vert #1 
    \right\vert\kern-0.25ex\right\vert\kern-0.25ex\right\vert}}

\renewcommand{\Im}{\mathfrak{Im}}
\renewcommand{\Re}{\mathfrak{Re}}
\renewcommand{\tilde}{\widetilde}
\renewcommand{\hat}{\widehat}
\renewcommand{\emptyset}{\varnothing}
\renewcommand{\rho}{\varrho}
\newcommand{\Id}{\textup{Id}}
\newcommand{\om}{\omega}
\newcommand{\ee}{\mathrm{e}}
\newcommand{\e}{\varepsilon}
\newcommand{\mdim}{\delta}

\newcommand{\eigenf}{h}

\newcommand{\z}{\zeta}
\newcommand{\aaa}{a}
\newcommand{\PF}{\mathcal L}
\newcommand{\Laplace}{L}
\newcommand{\prob}{\Pi}
\newcommand{\im}{\mathbf{i}}

\newcommand{\renfcn}{f}
\newcommand{\codefun}{\xi}
\newcommand{\codefunc}{\chi}
\newcommand{\1}{\mathbbm{1}}

\begin{document}

\title[Complex Ruelle-Perron-Frobenius thm for infinite Markov shifts]{A complex Ruelle-Perron-Frobenius theorem for infinite Markov shifts with applications to renewal theory}

\author{Marc Kesseb\"ohmer}  
\address{FB03 -- Mathematik und Informatik, Universt\"at Bremen, Bibliothekstr. 1, 28359 Bremen, Germany}
\email{mhk@math.uni-bremen.de}
\author{Sabrina Kombrink}
\address{ Universit\"at zu L\"ubeck,
Institut f\"ur Mathematik,
Ratzeburger Allee 160,
23562 L\"ubeck, Germany}
\email{kombrink@math.uni-luebeck.de}

\begin{abstract}
We prove a complex Ruelle-Perron-Frobenius theorem for Markov shifts over an infinite alphabet, whence extending results by M.~Pollicott from the finite to the infinite alphabet setting. As an application we obtain an extension of renewal theory in symbolic dynamics, as developed by S.~P.~Lalley and in the sequel generalised by the second author, now covering the infinite alphabet case.  
\end{abstract}
\dedicatory{Dedicated to the memory of our good friend and colleague \\Bernd O. Stratmann (1957-2015) }
 \maketitle

\section{Introduction}  

The core part of the present article is an extension of M.~Pollicott's results \cite{zbMATH03919132} concerning  spectral properties of Perron-Frobenius operators for complex potential functions to the setting of an infinite alphabet (see  {Thm.~\ref{thm:spectrum_regular} in} Sec.~\ref{sec:spectral}).
In order to obtain this extension we heavily make use of results on the Perron-Frobenius operator for real potential functions and infinite alphabets mainly developed by D.~Mauldin and M.~Urba\'nski in \cite{1033.37025} (see Sec.~\ref{sec:MU}).
Moreover, in Sec.~\ref{sec:analytic} we prove analyticity results for complex perturbed resolvents of Perron-Frobenius operators.

Applying M.~Pollicott's complex Ruelle-Perron-Frobenius theorem from \cite{zbMATH03919132} (see also \cite{ParryPollicott}) has lead to various new results, for instance to S.~P.~Lalley's renewal theorems for counting measures in symbolic dynamics \cite[Thms.~1 and 2]{Lalley}. 
In \cite{renewal} S.~P.~Lalley's ideas were generalised to more general measures. By this a setting was found which extends and unifies the setting of several established renewal theorems, namely
\begin{inparaenum}
	\item the above-mentioned theorems by S.~P.~Lalley \cite[Thms.~1 and 2]{Lalley}
	\item the classical key renewal theorem for finitely supported probability measures \cite{Feller2} and
	\item a class of Markov renewal theorems (see e.\,g. \cite{Alsmeyer,Asmussen}).
\end{inparaenum}
 By applying our new Thm.~\ref{thm:spectrum_regular} in Sec.~\ref{sec:renewal} we extend the setting of \cite{renewal} further by lifting the results from a finite to a countably infinite alphabet  leading to a new renewal theorem (see Thm.~\ref{thm:RT1}). 
 This 
\begin{inparaenum}
	\item exhibits new results in the vein of \cite{Lalley}
	\item encompasses the key renewal theorem for arbitrary discrete measures, see Cor.~\ref{cor:keyrt} and
	\item comprises certain Markov renewal theorems.
\end{inparaenum}

Renewal theorems are a useful tool in various areas of mathematics. Of particular interest to us are their applications in geometry (see e.g. \cite{Lalley:88} or \cite[Ch.~7]{Falconer:97}).  Indeed, the new renewal theorem, Thm.~\ref{thm:RT1}, that is stated and proved in Sec.~\ref{sec:renewal} allows for new results in this area. For instance, it yields statements concerning Minkowski measurability of limit sets of infinitely generated conformal graph directed systems (cGDS).
These results will be presented in a forthcoming article \cite{KK15b} by the authors. For some previous results on the finite alphabet case we refer to \cite{KK12,KK15}. The class of limit sets of infinitely generated cGDS is very rich and contains the boundary of Apollonian circle packings, limit sets of Fuchsian and Kleinian groups, self-similar and self-conformal sets and restricted continued fraction sets.

\section{Complex Ruelle-Perron-Frobenius theorem}\label{sec:symbolicrenewal}

In \cite{zbMATH03919132} a Ruelle-Perron-Fro\-benius theorem for complex potential functions was proven for the case that the underlying alphabet is finite.
The aim of this section is to extend these results from \cite{zbMATH03919132} to the setting of an infinite alphabet and to obtain analytic properties of resolvents which are associated to Perron-Frobenius operators for a family of complex potential functions. 
In Sec.~\ref{sec:MU1} we introduce the relevant notions and the central object, namely the complex Perron-Frobenius operator.
Important results concerning the Perron-Frobenius operator for real potential functions in the setting of an infinite alphabet have been obtained
by D.~Mauldin and M.~Urba\'nski and we collect their relevant results in Sec.~\ref{sec:MU}. 
In Sec.~\ref{sec:spectral} we use these statements to extend the findings of \cite{zbMATH03919132} to the
setting of an infinite alphabet, where we gain information on the spectrum of Perron-Frobenius operators $\mathcal L_{z\xi+\eta}$ for a family $(z\xi+\eta\mid z\in\mathbb C)$ of complex potential functions (with real-valued potentials $\xi,\eta$). At this point we would like to thank
Mariusz Urba\'nski for very valuable discussions on this problem. 
Finally, in Sec.\ref{sec:analytic}, we use the statements of Sec.~\ref{sec:MU}, \ref{sec:spectral} to obtain analytic properties of the resolvent-valued map $z\mapsto(\Id -\mathcal L_{z\xi+\eta})^{-1}$ with $\Id $ denoting the identity operator.

\subsection{The complex Ruelle-Perron-Frobenius operator}
\label{sec:MU1}

In the sequel $I\subset\mathbb{N}$ shall denote an at most countable
alphabet, $A\colon I\times I\to\{0,1\}$ an incidence matrix
and 
\[
E^{\infty}\defeq\{\omega\in I^{\mathbb N}\mid A_{\omega_{j}\omega_{j+1}}=1\ \text{for all}\ j\geq1\}
\]
the space of \emph{$A$-admissible infinite sequences}. $E^{n}$ denotes
the set of all subwords of $E^{\infty}$ of length $n\geq1$. The
space of \emph{$A$-admissible finite sequences} is denoted by 
\[
E^{*}\defeq\bigcup_{n\in\mathbb{N}_{0}}E^{n},
\]
where $E^{0}$ denotes the set which solely contains the empty word $\emptyset$.
For \linebreak $\omega=\omega_1\omega_2\cdots\in E^{\infty}$ and $n\in\mathbb N$ we write $\omega\vert_n\defeq\omega_1\cdots\omega_n$ for the initial subword of $\omega$ of length $n$.
For $\omega, x\in E^{\infty}$ we write $\omega\wedge x\defeq\max\{m\geq 0\mid \omega_i=x_i\ \text{for}\ i\leq m\}$ for the length of the longest common initial block of $\omega$ and $x$.

Throughout this paper we assume that the incidence
  matrix $A$ is  \emph{finitely irreducible}, that is there
exists a finite set $\Lambda\subset E^{*}$ such that for all $i,j\in I$
there is an $\omega\in\Lambda$ with $i\omega j\in E^{*}$. 
Note, finitely irreducible is a weaker condition than finitely primitive which is equivalent to the {\em big images and preimages (BIP) property}  of \cite{Sarig}, whenever the shift-dynamical system $(E^{\infty},\sigma)$ (with  $\sigma$  defined next)  is  topologically mixing.
On $E^{\infty}\cup E^{*}$
the \emph{shift map} $\sigma$ is defined by 
\[
\sigma(\omega)\coloneqq\begin{cases}
\omega_{2}\omega_{3}\cdots & \colon\omega=\omega_{1}\omega_{2}\cdots\in E^{\infty}\\
\omega_{2}\omega_{3}\cdots\omega_{n} & \colon\omega=\omega_{1}\omega_{2}\cdots\omega_{n}\in E^{n},\ n\geq2\\
\emptyset & \colon\omega\in E^{0}\cup E^{1}.
\end{cases}
\]
For $\omega\in E^{n}$ we denote the \emph{$\omega$-cylinder set} by
\[
\left[\omega\right]\defeq\left\{ x\in E^{\infty}\mid  x_{i}=\omega_{i}\ \forall i\in\left\{1,\ldots,n\right\} \right\} .
\]

The \emph{topological pressure function} of $u\colon E^{\infty}\to\mathbb{R}$ with respect to the shift map $\sigma\colon E^{\infty}\to E^{\infty}$ is defined by the well-defined limit
\[
P(u)\defeq\lim_{n\to\infty}\frac{1}{n}\log\sum_{\omega\in E^{n}}\exp\left(\sup_{\tau\in[\omega]}S_{n}u(\tau)\right),
\]
where
\[
S_{n}f\defeq\sum_{j=0}^{n-1}f\circ\sigma^{j}\ \text{for}\ n\geq1\quad\text{and}\quad S_{0}f\defeq0
\]
denotes the \emph{$n$-th Birkhoff sum} of $f\colon E^{\infty}\to\mathbb C$. Note that since the incidence matrix is finitely irreducible we have that the pressure defined above coincides with the Gurevich pressure  (cf.  \cite{HanusUrbanski,JKL, Sarig}).

We equip $I^{\mathbb N}$ with the product topology of the discrete topologies on $I$ and equip $E^{\infty}\subset I^{\mathbb N}$ with the subspace topology. By $\mathcal C(E^{\infty})$ resp.\ $\mathcal C(E^{\infty},\mathbb R)$ we denote the set of continuous compex- resp.\ real-valued functions on $E^{\infty}$. We refer to functions from $\mathcal C(E^{\infty})$ as \emph{potential functions}. The set of bounded continuous functions in $\mathcal C(E^{\infty})$ resp. $\mathcal C(E^{\infty},\mathbb R)$ with respect to the supremum-norm $\|\cdot\|_{\infty}$ is denoted by $\mathcal C_{b}(E^{\infty})$ resp. $\mathcal C_{b}(E^{\infty},\mathbb R)$. Of particular importance to us is the subclass of H\"older continuous functions.
\begin{defn}[H\"older continuity]
	For $f\in\mathcal C(E^{\infty})$, $\theta\in(0,1)$ and $n\in\mathbb N$ define
	\begin{align*}
		\text{var}_n(f)&\defeq \sup\{\lvert f(x)-f(y)\rvert\mid x,y\in E^{\infty}\ \text{and}\ x_i=y_i\ \text{for}\ i\leq n\},\\
		\|f\|_{\theta}&\defeq \sup_{n\geq 1}\frac{\text{var}_n(f)}{\theta^n}\quad\text{and}\\
		\mathcal F_{\theta}(E^{\infty})&\defeq\{f\in\mathcal C(E^{\infty})\mid \|f\|_{\theta}<\infty\}.
	\end{align*}
	Elements of $\mathcal F_{\theta}(E^{\infty})$ are called \emph{$\theta$-H\"older continuous} functions on $E^{\infty}$. Moreover, by  $\mathcal F_{\theta}(E^{\infty},\mathbb R)$ we denote the subclass of real-valued $\theta$-H\"older continuous functions on $E^{\infty}$.
	 {Note that by our definition a H\"older continuous function is not neccessarily bounded. For the respective spaces of bounded H\"older continuous functions we write $\mathcal F_{\theta}^b(E^{\infty})\defeq \mathcal F_{\theta}(E^{\infty}) \cap \mathcal C_b(E^{\infty})$ and
	$\mathcal F_{\theta}^b(E^{\infty},\mathbb R)\defeq \mathcal F_{\theta}(E^{\infty},\mathbb R) \cap \mathcal C_b(E^{\infty},\mathbb R)$.}
\end{defn}

In order to define the central object of this section, namely the Perron-Frobenius operator of a complex potential function $f=u+\mathbf{i}v\in\mathcal F_{\theta}(E^{\infty})$, we need to assume that 
\begin{equation}\label{eq:summable}
C_{u}\defeq  \sum_{e\in I}\exp(\sup(u\vert_{[e]}))<\infty.
\end{equation}
A function $u\in\mathcal{F}_{\theta}(E^{\infty},\mathbb R)$ which satisfies \eqref{eq:summable} is called \emph{summable}. 
Notice, when we write $f=u+\mathbf{i}v$ for $f\colon E^{\infty}\to\mathbb C$ we implicitely assume that $u$ and $v$ are real-valued.

\begin{rem}\label{rem:pressurebounded}
If $u\colon E^{\infty}\to\mathbb R$ is H\"older continuous then $u$ is summable if and only if $P(u)<\infty$, see \cite[Prop.~2{.}1{.}9]{1033.37025}. Moreover, if $u$ is summable, then $P(u)>-\infty$, which is a consequence of \cite[Thm.~2{.}1{.}5]{1033.37025}.
\end{rem}

\begin{defn}[Perron-Frobenius operator]
Let $f=u+\mathbf{i}v\in\mathcal F_{\theta}(E^{\infty})$ with summable $u$. 
The \emph{Perron-Frobenius-Operator} $\mathcal{L}_{f}\colon\mathcal{C}_{b}\left(E^{\infty}\right)\to\mathcal{C}_{b}\left(E^{\infty}\right)$
for the potential function $f$ acting on $\mathcal{C}_{b}\left(E^{\infty}\right)$ is defined by 
\[
\mathcal{L}_{f}(g)(\omega)
\defeq\sum_{e\in I:A_{e\omega_{1}}=1}\textup{e}^{f(e\omega)}g(e\omega)
=\sum_{y:\sigma y=\omega}\textup{e}^{f(y)}g(y).
\]
\end{defn}
The conjugate operator $\mathcal{L}_{f}^{*}$ acting on  $\mathcal{C}_{b}^{*}\left(E^{\infty}\right)$ can be restricted to the subset of finite Borel measures.  In fact, for any  finite
Borel measure $\mu$ the functional $\mathcal{L}_{f}^{*}\left(\mu\right)$ given by
\[
\mathcal{L}_{f}^{*}(\mu)(g)=\mu(\mathcal{L}_{f}(g))=\int\mathcal{L}_{f}(g)\textup{d}\mu,
\]
for all $g\in \mathcal{C}_{b}\left(E^{\infty}\right)$, is tight in the following sense. 
For every  $\e>0$ there exists a compact set $K\subset E^{\infty}$ such that for all   $g\in\mathcal{C}_{b}(E^{\infty})$ with $0\leq g\leq \mathds{1}_{E^{\infty}\setminus K}$ we have  
\[
 \left|\mathcal{L}_{f}^{*}\left(\mu\right)(g)\right|\leq \e.
\]
Here, $\mathds 1_B$ denotes the indicator function on a set $B$, that is $\mathds 1_B(x)=1$ if $x\in B$ and 0 otherwise.
To verify this condition we exclude the trivial measure and first choose an integer $M\in \mathbb{N}$  such that  $\varepsilon_{M}\defeq\sum_{e\in I,e>M}\exp(\sup(u|_{[e]}))\leq \e/(2 \mu (E^{\infty})).$
Since $E_{i,\ell}\defeq \{\omega\in E^{\infty}:\omega_{i}\geq \ell\}\downarrow \emptyset$, for $\ell\to \infty$ we find an increasing sequence $(\ell_{k})$ of integers with $\ell_{1}\geq M$ and $\mu (E_{k,\ell_{k}})<\e 2^{-k-1}/C_{u}$.  Then for the compact set $K\defeq E^{\infty}\setminus \bigcup_{k\in \mathbb N}E_{k,\ell_{k}}$ we have $\mu(E^{\infty}\setminus K)<\e /(2C_{u})$ and $e\omega\in K$ for all $e<M$ and $\omega\in K$ with $A_{e\omega_{1}}=1$. \label{p:Radon} For later use let us set $K_M\defeq K_M(\varepsilon,\mu)\defeq K$.
Hence we have  
\begin{align*}
  &\left\lvert\mathcal L_f^*(\mu)(g)\right\rvert
  \leq\int \mathcal{L}_{u}(g)\,\textup{d}\mu=\int_{K} \mathcal{L}_{u}(g)\,\textup{d}\mu
  +\int_{E^{\infty}\setminus K} \mathcal{L}_{u}(g)\,\textup{d}\mu  \\
 & \leq \int_{K} \sum_{e\in I:A_{e\omega_{1}}=1}\textup{e}^{u(e\omega)}g(e\omega)
  \,\textup{d}\mu(\omega) + C_{u}\mu(E^{\infty}\setminus K)
  \leq \e_M\mu(K)+\e/2\leq \e,
\end{align*}
 {which proves our claim of $\mathcal L_f^*(\mu)$ being tight. Applying an analogue of Riesz representation theorem for non-compact spaces stated in \cite[Thm.~7.10.6]{Bogachev} yields that} the functional $\mathcal{L}_{f}^{*}\left(\mu\right)$ can be represented uniquely by a finite Radon measure.

\subsection{A real Ruelle-Perron-Frobenius theorem and Gibbs measures}
\label{sec:MU}
A Bo\-rel probability measure $\mu$ on $E^{\infty}$ is said to be a \emph{Gibbs state} for $u\in\mathcal C(E^{\infty},\mathbb R)$ if there exists a constant $c>0$ such that 
\begin{equation}\label{eq:Gibbs}
 c^{-1}
 \leq\frac{\mu\left([\omega\vert_n]\right)}{\exp\left(S_n u(\omega)-n P(u)\right)}
 \leq c
\end{equation}
for every $\omega\in E^{\infty}$ and $n\in\mathbb N$.

The central theorem of this subsection, Thm.~\ref{thm:Ruelleinfinite}, is a combination of Lem.~2{.}4{.}1, Thms.~2{.}4{.}3, 2{.}4{.}6 and Cor.~2{.}7{.}5 from \cite{1033.37025}.
Note that Thms.~2{.}4{.}3, 2{.}4{.}6 in \cite{1033.37025} are stated and proved under the hypothesis that the incidence matrix $A$ is finitely primitive. In fact, the assumption of finitely irreducible $A$ suffices which we indicate as follows. From   \cite[Thm.~2.3.5]{1033.37025} it follows under the assumption of finite irreducibility  that   any convergent subsequence  of $(n^{-1}\sum_{k=0}^{n-1}\mathcal{L}_{u-P(u)}^{k}(\mathds 1))_n$ in the proof of \cite[Thm.~2.4.3]{1033.37025} is uniformly bounded away from zero. Here, $\mathds 1\coloneqq \mathds 1_{E^{\infty}}$ denotes the constant one-function on $E^{\infty}$.
\begin{thm}[{Real Ruelle-Perron-Frobenius theorem for infinite alphabets,\linebreak \cite{1033.37025}}]\label{thm:Ruelleinfinite}
  Suppose that $u\in\mathcal F_{\theta}(E^{\infty},\mathbb{R})$ for some $\theta\in(0,1)$ is summable. Then $\mathcal L_u$ preserves the space $\mathcal F_{\theta}^b(E^{\infty},\mathbb R)$, i.\,e.\ $\mathcal L_u\vert_{\mathcal F_{\theta}^b(E^{\infty},\mathbb R)}\colon \mathcal F_{\theta}^b(E^{\infty},\mathbb R)\to\mathcal F_{\theta}^b(E^{\infty},\mathbb R)$. Moreover,  the following hold.
  \begin{enumerate}[label={{\rm (\roman*)}}]
    \item  There is a unique Borel probability eigenmeasure $\nu_{u}$ of the conjugate Per\-ron-Frobenius operator $\mathcal{L}_{u}^{*}$ and the corresponding eigenvalue is equal to $\textup{e}^{P(u)}$. Moreover, $\nu_u$ is a Gibbs state for $u$.
    \item\label{it:bdd} The operator $\mathcal L_u\vert_{\mathcal F_{\theta}^b(E^{\infty},\mathbb R)}$ has an eigenfunction $h_u$ which is bounded from above and which satisfies $\int h_u\textup{d}\nu_u=1$. Further,  there exists an $R>0$ such that $h_u\geq R$ on $E^{\infty}$.
    \item The function $u$ has a unique ergodic $\sigma$-invariant Gibbs state $\mu_{u}$. 
    \item\label{it:convergencetoef} 
    There exist constants $\overline M>0$ and $\gamma\in(0,1)$ such that for every\linebreak $g\in\mathcal F_{\theta}^b(E^{\infty},\mathbb R)$ and every $n\in\mathbb N_0$
    \begin{equation}\label{eq:convergencetoef}
    \left\lVert\textup{e}^{-nP(u)}\mathcal{L}_{u}^{n}(g)-\int g\textup{d}\nu_{u}\cdot h_u\right\rVert_{\theta} \leq\overline{M}\gamma^{n}\left(\| g\|_{\theta}+\|g\|_{\infty}\right).
    \end{equation}
  \end{enumerate}
\end{thm}
 {Note that for our purposes it is important in Thm.~\ref{thm:Ruelleinfinite}\ref{it:bdd} that $h_u$ is uniformly bounded from below by $R$ rather than just positive, see e.\,g.\ proof of Prop.~\ref{prop:regular}.}
Directly from \eqref{eq:convergencetoef} we infer the following:
\begin{cor}\label{cor:disc-real}
  In the setting of Thm.~\ref{thm:Ruelleinfinite}{\rm \ref{it:convergencetoef}}, $\textup{e}^{P(u)}$ is a simple isolated eigenvalue of $\mathcal{L}_{u}\vert_{\mathcal F_{\theta}^b(E^{\infty},\mathbb R)}$. The rest of the spectrum of $\mathcal{L}_{u}\vert_{\mathcal F_{\theta}^b(E^{\infty},\mathbb R)}$  is contained in a disc centred at zero of radius at most $\gamma \ee^{P(u)}<\ee^{P(u)}$. 
\end{cor}

\subsection{Spectral Properties of the complex Ruelle-Perron-Frobenius operator and Complex Ruelle-Perron-Frobenius theorems}\label{sec:spectral}

Important spectral properties of the Perron-Frobenius operator in
the case of a finite alphabet have been obtained by W.~Parry and M.~Pollicott in \cite{ParryPollicott,zbMATH03919132}. In this section we are extending some of their
results to the setting of an infinite alphabet.
\begin{prop}
\label{prop:regular} Let $f=u+\mathbf{i}v\in\mathcal{F}_{\theta}(E^{\infty})$.
Suppose   that $u$ is summable.
For $0\leq a<2\pi$ the following are equivalent:
\begin{enumerate}[label={{\rm (\roman*)}}]
\item \label{it:complex_ev}$\textup{e}^{\mathbf{i}a+P(u)}$ is an eigenvalue
for $\mathcal{L}_{f}$.
\item \label{it:cohomologue} There exists $\zeta\in\mathcal{C}(E^{\infty},\mathbb{R})$
such that $v-a+\zeta\circ\sigma-\zeta\in\mathcal{C}(E^{\infty},2\pi\mathbb{Z})$, i.e. takes values only in $2\pi\mathbb{Z}$.
\end{enumerate}
\end{prop}
\begin{proof}
For the implication ``\ref{it:cohomologue} $\Rightarrow$ \ref{it:complex_ev}''
one readily sees that $\textup{e}^{\mathbf{i}a+P(u)}$ is an eigenvalue
corresponding to the eigenfunction $\textup{e}^{-\mathbf{i}\zeta}\eigenf_{u}$ with $\eigenf_u$ as in Thm.~\ref{thm:Ruelleinfinite}.

``\ref{it:complex_ev} $\Rightarrow$ \ref{it:cohomologue}'': We
can find $h\in\mathcal{C}_{b}(E^{\infty})$ such that 
\begin{equation}
\textup{e}^{\mathbf{i}a+P(u)}h=\mathcal{L}_{f}h.\label{eq:eigeneq}
\end{equation}
We write $h=|h|\exp(\mathbf{i}\tilde{h})$, where $\tilde{h}$ is
continuous (note, $\widetilde{h}$ is unique only up to $\textup{mod}(2\pi)$)
and thus we have for all $x\in E^{\infty}$ 
\begin{align}
\textup{e}^{\mathbf{i}a+P(u)}|h(x)|\textup{e}^{\mathbf{i}\tilde{h}(x)} & =\sum_{y:\sigma y=x}\textup{e}^{(u+\mathbf{i}(v+\tilde{h}))(y)}|h(y)|\nonumber \\
\Leftrightarrow\qquad\qquad\textup{e}^{P(u)}|h(x)| & =\sum_{y:\sigma y=x}\textup{e}^{(u+\mathbf{i}(v+\tilde{h}-\tilde{h}\circ\sigma-a))(y)}|h(y)|\label{eq:h1}
\end{align}
From \eqref{eq:eigeneq} and Thm.~\ref{thm:Ruelleinfinite}  we infer
\begin{align*}
\textup{e}^{P(u)}\|h\|_{L_{\nu_{u}}^{1}} & =\int\lvert\mathcal{L}_{f}h\rvert\textup{d}\nu_{u}\leq\int\mathcal{L}_{u}|h|\textup{d}\nu_{u}=\textup{e}^{P(u)}\|h\|_{L_{\nu_{u}}^{1}}
\end{align*}
Since $|\mathcal{L}_{f}h(x)|\leq\mathcal{L}_{u}|h|(x)$ for all $x\in E^{\infty}$
this shows that $\lvert\mathcal{L}_{f}h\rvert=\mathcal{L}_{u}|h|$
holds $\nu_{u}$-almost surely. Together with \eqref{eq:eigeneq}
we obtain
\begin{equation}
\textup{e}^{P(u)}|h|=\mathcal{L}_{u}|h|\qquad\nu_{u}-\text{almost surely}\label{eq:asequality}
\end{equation}
Thus, $|h|$ is a version of the unique strictly positive eigenfunction
$\eigenf_{u}$ of $\mathcal{L}_{u}$ to the eigenvalue $\textup{e}^{P(u)}$.
We deduce from Thm.~\ref{thm:Ruelleinfinite} that  {$\nu_{u}$-almost surely, $|h|\geq R$. (Here, it is important that $h_u\geq R>0$ holds true (see Thm.~\ref{thm:Ruelleinfinite}\ref{it:bdd}) rather than $h_u>0$.) Thus,} by Rem.~\ref{rem:pressurebounded} $\mathcal{L}_{u}|h|$
is $\nu_{u}$-almost surely bounded away from zero. The equations
\eqref{eq:h1} and \eqref{eq:asequality} together imply for $\nu_{u}$-almost
every $x\in E^{\infty}$ 
\begin{equation}
1=\sum_{y:\sigma y=x}\textup{e}^{\mathbf{i}(v+\tilde{h}-\tilde{h}\circ\sigma-a)(y)}\frac{\textup{e}^{u(y)}|h(y)|}{\mathcal{L}_{u}|h|(x)}.
\end{equation}
The above equation represents a (countable) convex combination of
points on the unit circle which lies on the unit circle. Thus, all
the points on the unit circle need to coincide. As moreover the left
hand side is equal to 1 it follows that 
\[
(v+\tilde{h}-\tilde{h}\circ\sigma-a)(y)\in2\pi\mathbb{Z}
\]
for all $y$ with $\sigma y=x$ and $\nu_{u}$-almost all $x\in E^{\infty}$.
Now, this set is dense in $E^\infty$, since being a Gibbs measure,  $\nu_{u}$ assigns positive mass to every cylinder set.   Since $v$ and $\tilde h$ are both continuous functions we obtain 
\[
v+\tilde{h}-\tilde{h}\circ\sigma-a\in\mathcal{C}(E^{\infty},2\pi\mathbb{Z}).
\]
\end{proof}
\begin{defn}
If $f=u+\mathbf{i}v\in\mathcal{F}_{\theta}(E^{\infty})$
satisfies one (and hence both) of the conditions of Prop.~\ref{prop:regular} then $f$ is called an \emph{$a$-function}. If $f$ is not an $a$-function (for any $a$) then $f$ is called \emph{regular}. 
\end{defn}
With the above proposition we conclude with the same argument as in
\cite[p.139]{zbMATH03919132} that the spectrum of $\mathcal{L}_{f}$
is precisely the spectrum of $\mathcal{L}_{u}$ rotated through the
angle $a$, when $f$ is an $a$-function. Together with Cor.~\ref{cor:disc-real}
this yields the following analogue of \cite[Prop.~3]{zbMATH03919132}. 
\begin{prop}[Complex Ruelle-Perron-Frobenius theorem for $a$-functions]\label{prop:spectrum_a_fcn}
If  $f=u+\mathbf{i}v\in\mathcal{F}_{\theta}(E^{\infty})$ with summable $u$ is an $a$-function then $\exp(\mathbf{i}a+P(u))$
is a simple eigenvalue for $\mathcal{L}_{f}$ and the rest of the
spectrum is contained in a disc of radius strictly smaller than $|\exp(\mathbf{i}a+P(u))|=\exp(P(u))$. 
\end{prop}
Now, we study the spectrum of $\mathcal{L}_{f}$ when $f$ is regular
and show that it is disjoint from the circle  {with centre at the origin and radius $\exp(P(u))$}. For this, we adapt
the arguments in \cite[p.~139]{zbMATH03919132}. 
For simplicity, we will often assume that $u$ is normalised so that $P(u)=0$ and $\mathcal L_u\mathds 1=\mathds 1$. This is possible, since for any summable $u\in\mathcal F_{\theta}(E^{\infty},\mathbb R)$ we have $P(u)\in\mathbb R$ (see Rem.~\ref{rem:pressurebounded}), whence $P(u-P(u))=0$ and $\mathcal L_{u-P(u)+\log \eigenf_u-\log\eigenf_u\circ\sigma}\1=\1$. Note that $u-P(u)+\log \eigenf_u-\log\eigenf_u\circ\sigma$ is summable when $u$ is summable by the bounded distortion property stated next.
\begin{lem}[cf.\ {\cite[Lem.~2{.}3{.}1]{1033.37025}}]\label{lem:BD}
 If $f\in\mathcal F_{\theta}(E^{\infty})$ then for all $n\in\mathbb N$, $\omega\in E^n$ and $x,y\in E^{\infty}$ with $\omega x,\omega y\in E^{\infty}$ we have
 \[
  \left\lvert S_n f(\omega x) - S_n f(\omega y) \right\rvert
  \leq\frac{\|f\|_{\theta}}{1-\theta}\theta^{n+1}.
 \]
\end{lem}

\begin{lem}\label{lem:bounded}
For $f=u+\mathbf{i}v\in\mathcal{F}_{\theta}(E^{\infty})$ with summable $u$
(with $P(u)=0$ and $\mathcal{L}_{u}\mathds{1}=\mathds{1}$) there exists
a constant $C=C(f,\theta)$ such that for all $n\in\mathbb{N}$ and\linebreak
$g\in\mathcal{F}_{\theta}^b(E^{\infty})$ we have 
\[
\|\mathcal{L}_{f}^{n}g\|_{\theta}\leq C\|g\|_{L_{\nu_{u}}^{1}}+(C+1)\|g\|_{\theta}\theta^{n}.
\]
\end{lem}
\begin{proof}
For $x,y\in E^{\infty}$ satisfying $x\wedge y=m\geq 1$ 
we have $\{\omega\in E^{n}\mid\omega x\in E^{\infty}\}=\{\omega\in E^{n}\mid\omega y\in E^{\infty}\}$.
Thus, 
\begin{align*}
\lefteqn{\lvert\mathcal{L}_{f}^{n}g(x)-\mathcal{L}_{f}^{n}g(y)\rvert}\\
 & =\left\lvert \sum_{\omega\in E^{n},\omega x\in E^{\infty}}\left(\textup{e}^{S_{n}f(\omega x)}g(\omega x)-\textup{e}^{S_{n}f(\omega y)}g(\omega y)\right)\right\rvert \\
 & \leq\sum_{\omega\in E^{n},\omega x\in E^{\infty}}\left\lvert \textup{e}^{S_{n}f(\omega x)}-\textup{e}^{S_{n}f(\omega y)}\right\rvert \cdot|g(\omega x)|+\left\lvert \textup{e}^{S_{n}f(\omega y)}\right\rvert \cdot\lvert g(\omega x)-g(\omega y)\rvert\\
 & \stackrel{(*)}{\leq}\sum_{\omega\in E^{n}\omega x\in E^{\infty}}\left(\textup{e}^{S_{n}u(\omega z)}\lvert S_{n}f(\omega x)-S_{n}f(\omega y)\rvert\cdot|g(\omega x)|+\textup{e}^{S_{n}u(\omega y)}\lvert g(\omega x)-g(\omega y)\rvert\right)\\
 & \leq\sum_{\omega\in E^{n},\omega x\in E^{\infty}}\left(\textup{e}^{S_{n}u(\omega z)}\|f\|_{\theta}\frac{\theta^{m+1}}{1-\theta}|g(\omega x)|+\textup{e}^{S_{n}u(\omega y)}\|g\|_{\theta}\theta^{n+m}\right)\\
 & \stackrel{(**)}{\leq}\left(\|f\|_{\theta}\frac{\theta^{m+1}}{1-\theta}\sum_{\omega\in E^{n},\omega x\in E^{\infty}}c\nu_{u}([\omega])|g(\omega x)|\right)+\|g\|_{\theta}\theta^{n+m}\\
 & \leq\|f\|_{\theta}\frac{c\theta^{m+1}}{1-\theta}\left(\int|g|\textup{d}\nu_{u}+\|g\|_{\theta}\theta^{n}\right)+\|g\|_{\theta}\theta^{n+m}.
\end{align*}
Here, $(*)$ follows from the mean value theorem with some $z\in E^{\infty}$ and in $(**)$ we used the Gibbs property of $\nu_u$ with constant $c$, see \eqref{eq:Gibbs} and Thm.~\ref{thm:Ruelleinfinite}.
Setting $C\defeq \|f\|_{\theta}\tfrac{c\theta}{1-\theta}$, 
we obtain 
\[
\textup{var}_{m}\left(\mathcal{L}_{f}^{n}g\right)\leq\theta^{m}\left(C\|g\|_{L_{\nu_{u}}^{1}}+(C+1)\theta^{n}\|g\|_{\theta}\right)
\]
which shows the assertion. 
\end{proof}
Choose a point $\textup{e}^{\mathbf{i}t}$ on the unit circle. For
$h\in\mathcal{F}_{\theta}(E^{\infty})$ which satisfies
\[
\vertiii{h}_{\theta}\defeq\|h\|_{L_{\nu_{u}}^{1}}+\|h\|_{\theta}\leq1
\]
 and for each $N\in\mathbb{N}$ we write 
\[
h_{N}\defeq\frac{1}{N}\sum_{n=0}^{N-1}\mathcal{L}_{f-\mathbf{i}t}^{n}h.
\]
Note that 
\begin{align*}
 \| h\|_{\infty}
 &=\int \| h\|_{\infty}\,\textup{d}\nu_u
 \leq \sum_{i\in I} \int_{[i]}\left(\lvert h(x)\rvert + \textup{var}_1(\lvert h\rvert)\right)\,\textup {d}\nu_u(x)
\\ &\leq \| h\|_{L_{\nu_u}^1}+\theta\| h \|_{\theta}
 \leq \vertiii{h}_{\theta}.
\end{align*}
Thus, $\vertiii{h}_{\theta}\leq1$ in particular implies $\|h\|_{\infty}\leq1$. 

\begin{lem}\label{lem:L1convergence}
Let $f=u+\mathbf{i}v\in\mathcal{F}_{\theta}(E^{\infty},\mathbb{C})$
be regular and $u$ summable (with $P(u)=0$ and $\mathcal{L}_{u}\mathds{1}=\mathds{1}$).
Then $\|h_{N}\|_{L_{\nu_{u}}^{1}}$ tends to zero when $N\to\infty$. \end{lem}
\begin{proof}
Suppose  {for a contradiction} that $\|h_{N}\|_{L_{\nu_{u}}^{1}}$ does not tend to zero
when $N\to\infty$. Under this assumption $\limsup_{N\to\infty}\|h_{N}\|_{L_{\nu_{u}}^{1}}\eqdef s\in(0,1]$.
Thus, there is a sequence $(N_{k})_{k\in\mathbb{N}}$ such that 
\[
\lim_{k\to\infty}\|h_{N_{k}}\|_{L_{\nu_{u}}^{1}}=s.
\]
We now show that there exists a subsequence of $(N_{k})_{k\in\mathbb{N}}$
along which $(h_{N})_{N}$ converges $\nu_{u}$-almost surely to a
function $h^{*}\in\mathcal{C}(E^{\infty})$. For this, we assume without loss of generality that $I=\mathbb N$ and need the following two ingredients. 
\begin{enumerate}
\item \label{it:Ausschoepfung} Recall that on p.~\pageref{p:Radon} we constructed for every Borel measure $\mu$ and every $\e>0$, $M\in\mathbb N$ a compact set $K=K_M=K_{M}(\e,\mu)$ for which $\mu(E^{\infty}\setminus K)<\e/(2C_u)$ and $e\om\in K$ for every $e<M$ and $\om\in K$ with $A_{e\om_1}=1$. Moreover, by construction $K_M\uparrow\bigcup_{M\in\mathbb N} K_M$ as $M\to\infty$. For $M\in \mathbb N$ we define the compact set $Y_M\defeq K_M(2^{-M+1},\nu_u)$.
Then
\begin{enumerate}
	\item $\displaystyle{ \nu_u(E^{\infty}\setminus Y_M)<2^{-M}/C_u}$
	\item $\displaystyle{\nu_u\left(\bigcup Y_M\right)=\lim_{M\to\infty}\nu_u(Y_M)=1}$
	\item $\displaystyle{e\om\in Y_M}$ for all $\displaystyle{\om\in Y_M}$ and $\displaystyle{e\in I}$ with $\displaystyle{e<M}$, $\displaystyle{A_{e\om_1}=1}$
	\item $\displaystyle{\{\om\in E^{\infty}\mid \om_k\leq M\ \forall\, k\in \mathbb N\}\subset Y_M}$.
\end{enumerate}
\item $(h_{N})_{N\in\mathbb{N}}$ is equicontinuous on $E^{\infty}$:
Let $x,y\in E^{\infty}$ be such that $x\wedge y=m\in\mathbb N$. By Lem.~\ref{lem:bounded}
there exists $C=C(f-\mathbf{i}t,\theta)$ such that 
\begin{align}
\lvert h_{N}(x)-h_{N}(y)\rvert & \leq\frac{1}{N}\sum_{n=0}^{N-1}\left(C\|h\|_{L_{\nu_{u}}^{1}}+(C+1)\|h\|_{\theta}\theta^{n}\right)\theta^{m}\nonumber \\
 & \leq\left(C+\frac{C+1}{N(1-\theta)}\right)\theta^{m},\label{eq:hNC}
\end{align}
which shows equicontinuity. 
\end{enumerate}
Because of the equicontinuity, Azel\`a-Ascoli implies that for every
$M\in\mathbb{N}$ we can find a subsequence $(N_{m}^{M})_{m\in\mathbb{N}}$
of $(N_{k})_{k\in\mathbb{N}}$ such that $(h_{N_{m}^{M}})_{m}$ converges
uniformly on $Y_{M}$ to a function $h_{M}^{*}\colon E^{\infty}\to\mathbb{C}$,
which restricted to $Y_{M}$ lies in $\mathcal{C}(Y_{M})$.
By construction, we can assume that $(N_{m}^{M+1})_{m}$ is a subsequence
of $(N_{m}^{M})_{m}$. Moreover, it is clear that $h_{M+1}^{*}|_{Y_{M}}=h_{M}^{*}|_{Y_{M}}$.
Because of \ref{it:Ausschoepfung} we can find to $\nu_{u}$-almost
every $x\in E^{\infty}$ an $M_{x}\in\mathbb{N}$ such that $x\in Y_{M_{x}}$.
Let $h^{*}(x)\defeq h_{M_{x}}^{*}(x)$. This defines a function $h^{*}$
on $\bigcup Y_{M}$ which lies in $\mathcal{C}(\bigcup Y_{M})$
and is even uniformly continous because of the following.
For $M\in\mathbb{N}$, $x,y\in Y_{M}$ 
\begin{align*}
&\lvert h^{*}(x)-h^{*}(y)\rvert \\& \leq\limsup_{m\to\infty}\lvert h^{*}(x)-h_{N_{m}^{M}}(x)\rvert+\lvert h_{N_{m}^{M}}(x)-h_{N_{m}^{M}}(y)\rvert+\lvert h_{N_{m}^{M}}(y)-h^{*}(y)\rvert\\
 & \leq C\cdot \theta^{x\wedge y}
\end{align*}
with $C=C(f-\mathbf{i}t,\theta)$ as in \eqref{eq:hNC}. As $\nu_{u}$ is a Gibbs state for $u$, it assigns
positive mass to every cylinder set. Thus, \ref{it:Ausschoepfung} implies that $\bigcup Y_{M}$ is dense in
$E^{\infty}$ and we can uniquely extend $h^{*}$ continuously to
$E^{\infty}$.
We denote the extension of $h^{*}$ to $E^{\infty}$ by $h^{*}$ as
well and show in the following that $h^{*}$ is a non-zero eigenfunction
of $\mathcal{L}_{f}$ to the eigenvalue $\textup{e}^{\mathbf{i}t}$,  {a contradiction to the regularity of $f$}.

\setcounter{enumi}{2} 
\begin{enumerate}
\item \label{it:hstarnonzero} $h^{*}$ is non-zero: For all $M\in\mathbb{N}$,
\begin{align*}
\|h^{*}\|_{L_{\nu_{u}}^{1}} & =\int_{Y_{M}}\lvert h^{*}\rvert\,\textup{d}\nu_{u}+\int_{E^{\infty}\setminus Y_{M}}\lvert h^{*}\rvert\,\textup{d}\nu_{u}\\
 & =\lim_{m\to\infty}\int_{Y_{M}}\lvert h_{N_{m}^{M}}\rvert\,\textup{d}\nu_{u}+\int_{E^{\infty}\setminus Y_{M}}\lvert h^{*}\rvert\,\textup{d}\nu_{u}\\
 & \geq\limsup_{m\to\infty}\left(\|h_{N_{m}^{M}}\|_{L_{\nu_{u}}^{1}}-\int_{E^{\infty}\setminus Y_{M}}\lvert h_{N_{m}^{M}}\rvert\textup{d}\nu_{u}\right)\\
 & \geq s-\nu_u(E^{\infty}\setminus Y_M)
 > s-2^{-M}/C_u
\end{align*}
yielding 
\[
0<s\leq\|h^{*}\|_{L_{\nu_{u}}^{1}}.
\]
\item \label{it:hstarbd} $\|h^*\|_{\infty}\leq 1$.
\item \label{it:hstaref} To show that $h^{*}$ is an eigenfunction to the
eigenvalue $\textup{e}^{\mathbf{i}t}$, we first take $x\in\bigcup_{m\in\mathbb{N}}Y_{m}$,
say $x\in Y_{M}$, $M\in \mathbb{N}$.
Further, recall that for all $x\in E^{\infty}$, \[\e_M\defeq\sum_{\genfrac{}{}{0pt}{1}{e\in I}{ e>M}}\exp(\sup u\vert_{[e]}) \geq\sum_{\genfrac{}{}{0pt}{1}{y\in E^{\infty}\setminus Y_M}{ \sigma y=x}}\exp(u(y)).\] 
  Then, 
\begin{align*}
 & \lvert\mathcal{L}_{f}h^{*}(x)-\textup{e}^{\mathbf{i}t}h^{*}(x)\rvert\\
 & \quad\leq\left\lvert \sum_{y:\sigma y=x,y\in Y_{M}}\textup{e}^{f(y)}h^{*}(y)-\textup{e}^{\mathbf{i}t}h^{*}(x)\right\rvert +\varepsilon_{M}\\
 & \quad=\left\lvert \lim_{m\to\infty}\frac{1}{N_{m}^{M}}\sum_{n=0}^{N_{m}^{M}-1}\left(\mathcal{L}_{f-\mathbf{i}t}^{n+1}h(x)\textup{e}^{\mathbf{i}t}-\sum_{ \genfrac{}{}{0pt}{1}{\sigma y=x}{y\notin Y_{M}}}\textup{e}^{f(y)}\mathcal{L}_{f-\mathbf{i}t}^{n}h(y)\right)-\textup{e}^{\mathbf{i}t}h^{*}(x)\right\rvert +\varepsilon_{M}\\
 & \quad\leq\left\lvert \lim_{m\to\infty}h_{N_{m}^{M}}(x)\textup{e}^{\mathbf{i}t}+\frac{\textup{e}^{\mathbf{i}t}}{N_{m}^{M}}\left(\mathcal{L}_{f}^{N_{m}^{M}}h(x)-h(x)\right)-\textup{e}^{\mathbf{i}t}h^{*}(x)\right\rvert +2\varepsilon_{M}\\
 & \quad=2\varepsilon_{M}.
\end{align*}
As $x\in Y_{M}$ implies $x\in Y_{\overline{M}}$ for all $\overline{M}\geq M$
we conclude that 
\[
\mathcal{L}_{f}h^{*}(x)=\textup{e}^{\mathbf{i}t}h^{*}(x)
\]
for all $x\in\bigcup_{M\in\mathbb{N}}Y_{M}$. Since $\mathcal{L}_{f}h^{*}$
and $\textup{e}^{\mathbf{i}t}h^{*}$ are both uniformly continuous
and as they coincide on a dense subset of $E^{\infty}$, they need
to coincide on $E^{\infty}$, that is 
\[
\mathcal{L}_{f}h^{*}=\textup{e}^{\mathbf{i}t}h^{*}\qquad\text{on}\ E^{\infty}.
\]

\end{enumerate}
\ref{it:hstarnonzero} and \ref{it:hstarbd} together with \ref{it:hstaref} show that
$\textup{e}^{\mathbf{i}t}$ is an eigenvalue of $\mathcal{L}_{f}$
with non-zero eigenfunction $h^{*}\in\mathcal C_b(E^{\infty})$. This together with Prop.~\ref{prop:regular}
is a contradiction to the assumption that $f$ is regular, whence
\[
\lim_{n\to\text{\ensuremath{\infty}}}\|h_{N}\|_{L_{\nu_{u}}^{1}}=0.
\]
\end{proof}
\begin{lem}
 Let $f=u+\mathbf{i}v\in\mathcal{F}_{\theta}(E^{\infty})$
be regular and $u$ summable (with $P(u)=0$ and $\mathcal{L}_{u}\mathds{1}=\mathds{1}$).
Then $\|h_{N}\|_{\theta}$ tends to zero when $N\to\infty$. \end{lem}
\begin{proof}
For any $n\leq N$ we deduce the following inequalities from Lem.~\ref{lem:bounded}. 
\begin{align*}
\|h_{N}\|_{\theta} & \leq\left\lVert\frac{1}{N}\sum_{k=0}^{N-1}\mathcal{L}_{f-\mathbf{i}t}^{k}h-\mathcal{L}_{f-\mathbf{i}t}^{n}h_{N}\right\rVert_{\theta}+\|\mathcal{L}_{f-\mathbf{i}t}^{n}h_{N}\|_{\theta}\\
 & \leq\frac{1}{N}\sum_{k=0}^{n-1}\left(\|\mathcal{L}_{f-\mathbf{i}t}^{k}h\|_{\theta}+\|\mathcal{L}_{f-\mathbf{i}t}^{k+N}h\|_{\theta}\right)+\|\mathcal{L}_{f-\mathbf{i}t}^{n}h_{N}\|_{\theta}\\
 & \leq\frac{1}{N}\sum_{k=0}^{n-1}\left(C\|h\|_{L_{\nu_{u}}^{1}}\cdot2+(C+1)\|h\|_{\theta}\left(\theta^{k}+\theta^{k+N}\right)\right)\\
 & \hspace*{1cm}+C\|h_{N}\|_{L_{\nu_{u}}^{1}}+(C+1)\|h_{N}\|_{\theta}\theta^{n}\\
 & \leq\frac{2n}{N}C\|h\|_{L_{\nu_{u}}^{1}}+\frac{(C+1)\left(1+\theta^{N}\right)}{(1-\theta)N}\left(1-\theta^{n}\right)\|h\|_{\theta}+C\|h_{N}\|_{L_{\nu_{u}}^{1}}\\
 & \hspace*{1cm}+(C+1)\theta^{n}\left(C\|h\|_{L_{\nu_{u}}^{1}}+(C+1)\|h\|_{\theta}\frac{1-\theta^{N}}{N(1-\theta)}\right).
\end{align*}
The above inequalities are in particular valid for $n=\sqrt{N}$. Additionally using that $\vertiii{h}_{\theta}\leq 1$ we obtain 
\begin{align}
\|h_{N}\|_{\theta} 
 & \leq\frac{2C}{\sqrt{N}}+\frac{(C+1)\left(1+\theta^{N}\right)}{(1-\theta)N}\left(1-\theta^{\sqrt{N}}\right)+C\|h_{N}\|_{L_{\nu_{u}}^{1}}\nonumber\\
 &\hspace*{1cm}+(C+1)\theta^{\sqrt{N}}C+\frac{(C+1)^{2}\theta^{\sqrt{N}}}{N(1-\theta)}.\label{eq:proofconvergencetheta}
\end{align}
Applying Lem.~\ref{lem:L1convergence} now finishes the proof.
\end{proof}

\begin{lem}\label{lem:Hcompact}
 {If} $u$ is summable  {then} the closed unit ball \[H\defeq \{h\in\mathcal F_{\theta}(E^{\infty})\mid\vertiii{h}_{\theta}\leq1\}\]
is compact in the Banach space $\left(L_{\nu_{u}}^{1}(E^{\infty}),\|\cdot\|_{L_{\nu_{u}}^{1}}\right)$. 
\end{lem}

\begin{proof}
It suffices to show that $H$ is a sequentially compact subset of
the Banach space of $L_{\nu_{u}}^{1}(E^{\infty})$-functions. Since $H$ is equicontinuous,
we can use the arguments of the proof of Lem.~\ref{lem:L1convergence}
to conclude that any sequence $(f_{n})_{n\in\mathbb{N}}$ in $H$
posesses a subsequence $\left(f_{n_{m}}\right)_{m\in\mathbb{N}}$
which converges pointwise to a uniformly continuous limiting function
$f^{*}$. We use the notation from Lem.~\ref{lem:L1convergence}.
Write $\delta_{m}(x)\defeq\lvert f^{*}(x)-f_{n_{m}}(x)\rvert$
and $\varepsilon_{m}^{M}\defeq\sup_{x\in Y_{M}}\lvert f^{*}(x)-f_{n_{m}}(x)\rvert$.
Uniform convergence of $(f_{n_m})_m$ on $Y_M$ implies $\lim_{m\to\infty}\e_m^M=0$.
  Further let $\varepsilon_{M}\defeq\sum_{e\in I,e>M}\exp(\sup(u\vert_{[e]}))$ be as above. We again assume without loss of generality that $I=\mathbb N$. Because of 
\begin{align*}
\lvert f^{*}(x)-f^{*}(y)\rvert & \leq\lvert f^{*}(x)-f_{n_{m}}(x)\rvert+\lvert f_{n_{m}}(x)-f_{n_{m}}(y)\rvert+\lvert f_{n_{m}}(y)-f^{*}(y)\rvert\\
 & \leq\delta_{m}(x)+\theta^{x\wedge y}\|f_{n_{m}}\|_{\theta}+\delta_{m}(y)
\end{align*}
and that $\|f_{n_m}\|_{\infty}$ implies $\|f^*\|_{\infty}\leq 1$, which gives
\begin{align*}
\|f^{*}-f_{n_{m}}\|_{L_{\nu_{u}}^{1}}
=\int_{Y_{M}}\lvert f^{*}-f_{n_{m}}\rvert\textup{d}\nu_{u}+\int_{Y_{M}^{c}}\lvert f^{*}-f_{n_{m}}\rvert\textup{d}\nu_{u} \leq \varepsilon_{m}^{M}+ 2\cdot 2^{-M}/C_u
\end{align*}
we have 
\begin{align*}
& \vertiii{f^{*}}_{\theta} 
 =\|f^{*}\|_{\theta}+\|f^{*}\|_{L_{\nu_{u}}^{1}}
=\sup_{e\in I, x\not=y\in [e]}\frac{\lvert f^{*}(x)-f^{*}(y)\rvert}{\theta^{x\wedge y}}+\|f^{*}\|_{L_{\nu_{u}}^{1}}\\
 & \leq\sup_{e\in I, x\not=y\in [e]}\limsup_{M\to\infty}\limsup_{m\to\infty}\frac{\delta_{m}(x)+\delta_{m}(y)}{\theta^{x\wedge y}}+\|f_{n_{m}}\|_{\theta}+\varepsilon_{m}^{M}+\frac{2^{-M+1}}{C_u}+\|f_{n_{m}}\|_{L_{\nu_{u}}^{1}}\\
 & \leq1+\sup_{e\in I, x\not=y\in [e]}\limsup_{M\to\infty}\limsup_{m\to\infty}\frac{\delta_{m}(x)+\delta_{m}(y)}{\theta^{x\wedge y}}+\varepsilon_{m}^{M}+2^{-M+1}/C_u\\
 & =1.
\end{align*}
\end{proof}

Combining Lem.~\ref{lem:L1convergence}, \ref{lem:Hcompact} and \eqref{eq:proofconvergencetheta} yields that we can choose $N\in\mathbb N$ such that $\vertiii{h_N}_{\theta}<1$ for all $h\in H$ with $H$ defined in Lem.~\ref{lem:Hcompact}.
The arguments of \cite[p.~140]{zbMATH03919132} imply that $\textup{e}^{\mathbf{i} t}$ is not in the spectrum of $\mathcal L_f$ and finally the following theorem.
\begin{thm}[Complex Ruelle-Perron-Frobenius theorem for regular functions]\label{thm:spectrum_regular}
Let $f=u+\mathbf{i}v\in\mathcal{F}_{\theta}(E^{\infty})$.
Suppose that  $u$ is summable. 
If $f$ is regular then the spectrum of $\mathcal{L}_{f}$ is contained
in a disc of radius strictly smaller than $\exp(P(u))$. 
\end{thm}

\subsection{Analyticity of $z\mapsto (\Id -\mathcal{L}_{\eta+z\xi})^{-1}$} \label{sec:analytic}

Consider $\eta,\xi\in\mathcal F_{\theta}(E^{\infty},\mathbb R)$ with $\xi\geq 0$. We suppose that there exists a unique $\delta\in\mathbb R$ for which $P(\eta-\delta\xi)=0$.
This condition in particular implies that $\xi$ cannot be identically zero. Moreover, by Rem.~\ref{rem:pressurebounded} 
\[
 t^*\defeq\sup\{t\in\mathbb R\mid \eta+ t\xi\ \text{is summable}\}
\]
satisfies $-\delta\leq t^*$.

 In order to study the analytic properties of the operator-valued function $z\mapsto \linebreak(\Id -\mathcal{L}_{\eta+z\xi})^{-1}$, we let $\mathcal B(\mathcal F_{\theta}^b(E^{\infty}))$ denote the set of all bounded linear operators on $\mathcal F_{\theta}^b(E^{\infty})$ to $\mathcal F_{\theta}^b(E^{\infty})$. We equip $\mathcal F_{\theta}^b(E^{\infty})$ with the norm $\lvert\cdot\rvert_{\theta}\defeq \|\cdot\|_{\theta}+\|\cdot\|_{\infty}$ and note that $(\mathcal F_{\theta}^b(E^{\infty}),\lvert\cdot\rvert_{\theta})$ is a Banach space, which makes $(\mathcal B(\mathcal F_{\theta}^b(E^{\infty})),\|\cdot\|_{\textup{op}})$ a Banach space. Here, $\|\cdot\|_{\textup{op}}$ denotes the operator norm.
 
We let
$G\defeq\{z\in\mathbb C\mid \Re(z)<t^*\}$
denote the open domain of parameters $z\in\mathbb C$ for which $\eta+\Re(z)\xi$ is summable.
As a consequence of \cite[Cor.~2{.}6{.}10]{1033.37025}, $z\mapsto\mathcal L_{\eta+ z\xi}\in \mathcal B(\mathcal F_{\theta}^b(E^{\infty}))$ is holomorphic in $G$.
Hence, if $(\Id -\mathcal L_{\eta+z\xi})^{-1}$ exists in some open domain $D\subseteq G$, then $z\mapsto(\Id -\mathcal L_{\eta+z\xi})^{-1}$ is holomorphic (see e.\,g. \cite[Ch.~1.4.5]{Kato}).
The map $t\mapsto P(\eta+t\xi)$ is monotonically increasing for $t\in\mathbb R$. Moreover, the uniqueness of $\delta\in\mathbb R$ with $P(\eta-\delta\xi)=0$ implies $P(\eta+t\xi)<0$ for $t<-\delta$. Therefore,
Prop.~\ref{prop:spectrum_a_fcn} and Thm.~\ref{thm:spectrum_regular}  imply that 1 does not lie in the spectrum of $\mathcal L_{\eta+z\xi}$ for $\Re(z)<-\delta$, which proves the following proposition.

\begin{prop}[cf. {\cite[Prop.~7{.}1]{Lalley}}]
 $z\mapsto (\Id -\mathcal{L}_{\eta+z\xi})^{-1}$ is holomorphic in the half-plane $\Re(z)<-\delta$.
\end{prop}

Moreover, the next proposition shows that $z\mapsto (\Id -\mathcal{L}_{\eta+z\xi})^{-1}$ has a meromorphic extension to a neighbourhood of $z=-\delta$, provided $-\delta<t^*$. The meromorphic extension builds on regular perturbation theory {\cite[Ch.~7 and 4{.}3]{Kato}} which shows that the functions $z\mapsto\gamma_{\eta+z\xi}\defeq\exp(P(\eta+z\xi))$, $z\mapsto\eigenf_{\eta+z\xi}$ and $z\mapsto\nu_{\eta+z\xi}$ extend to holomorphic functions in a neighbourhood of the half-line $(-\infty,t^*)$ such that $\gamma_{\eta+z\xi}\neq 0$, $\mathcal L_{\eta+z\xi}\eigenf_{\eta+z\xi}=\gamma_{\eta+z\xi}\eigenf_{\eta+z\xi}$, $\mathcal L^*_{\eta+z\xi}\nu_{\eta+z\xi}=\gamma_{\eta+z\xi}\nu_{\eta+z\xi}$ and $\nu_{\eta+z\xi}(\eigenf_{\eta+z\xi})=\nu_0(\eigenf_{\eta+z\xi})=1$ (see \cite[p.\,27]{Lalley}). Since furthermore, $P(\eta+t\xi)=\log(\gamma_{\eta+t\xi})$ for $t\in\mathbb R$, we can extend the topological pressure function analytically by setting $P(\eta+z\xi)\defeq\log(\gamma_{\eta+z\xi})$, though formally, the definition can only be made modulo $2\pi\mathbf{i}$ (see \cite[p.\,31]{ParryPollicott}).
Under the assumption that $-\delta< t^*$, the proof of the existence of the meromorphic extension given in \cite[Prop.~7{.}2]{Lalley} remains valid in the setting of an infinite alphabet, when using Thm.~\ref{thm:Ruelleinfinite} and Cor.~\ref{cor:disc-real} instead of \cite[Thm.~A]{Lalley}. Thus, we obtain the following.

\begin{prop}[cf.\ {\cite[Prop.~7{.}2]{Lalley}}]\label{prop:Lalley:7.2}
 If $-\delta< t^*$ then $z\mapsto(\Id -\mathcal L_{\eta+z\xi})^{-1}$ has a simple pole at $z=-\delta$. In particular, for each $\chi\in\mathcal F_{\theta}^b(E^{\infty})$
 \begin{equation}\label{eq:meromorphic}
  (\Id -\mathcal L_{\eta+z\xi})^{-1}\chi
  =\frac{\textup{e}^{P(\eta+ z\xi)}}{1-\textup{e}^{P(\eta+ z\xi)}}\nu_{\eta+z\xi}(\chi)\eigenf_{\eta+z\xi} + (\Id -\mathcal L''_{\eta+z\xi})^{-1}\chi
 \end{equation}
  for $z$ in some punctured neighbourhood of $z=-\delta$. Here,
  \begin{equation*}
  \begin{array}{rcll}
   \mathcal L''_{\eta+z\xi} & \defeq &\mathcal L_{\eta+z\xi} - \mathcal L'_{\eta+z\xi}&\text{with}\\
   \mathcal L'_{\eta+z\xi}\chi & \defeq &\textup{e}^{P(\eta+ z\xi)}\nu_{\eta+z\xi}(\chi)\eigenf_{\eta+z\xi} &\text{for}\ \chi\in\mathcal F_{\theta}^b(E^{\infty}).
  \end{array}
  \end{equation*}
\end{prop}
The factor $\textup{e}^{P(\eta+ z\xi)}$ of the first summand of \eqref{eq:meromorphic} is missing in \cite{Lalley}. However, this does not make a difference, since the $z$-value of interest is $z=-\delta$, where $P(\eta+ z\xi)=0$.

Under the additional assumption that $\int-(\eta+t\xi)\textup{d}\mu_{\eta-\delta\xi}<\infty$ for all $t$ in an open neighbourhood of $-\delta$, \cite[Prop.~2{.}6{.}13]{1033.37025} gives
\begin{equation}\label{eq:pressure_derivative}
 \frac{\textup{d}}{\textup{d}t}P(\eta+t\xi)\bigg\vert_{t=-\delta}=\int\xi\textup{d}\mu_{\eta-\delta\xi}>0,
\end{equation}
since $\eta-\delta\xi$ is summable.
Combining the above-stated results from regular perturbation theory, Prop.~\ref{prop:Lalley:7.2} and \eqref{eq:pressure_derivative} we obtain the following.
\begin{cor}\label{cor:residue}
 If $-\delta< t^*$ and $\int-(\eta+t\xi)\textup{d}\mu_{\eta-\delta\xi}<\infty$ for all $t$ in an open neighbourhood of $-\delta$ then, for each  $\chi\in\mathcal F_{\theta}(E^{\infty},\mathbb R)$ and $x\in E^{\infty}$,
 we have that  the residue of $z\mapsto(\Id -\mathcal L_{\eta+z\xi})^{-1}\chi(x)$ at the simple pole $z=-\delta$ is equal to 
 \[
  -\frac{\int \chi\textup{d}\nu_{\eta-\delta\xi}}{\int \xi\textup{d}\mu_{\eta-\delta\xi}}\eigenf_{\eta-\delta\xi}(x).
   \]
\end{cor}

  By using the same arguments as in \cite{Lalley} and applying Prop.~\ref{prop:spectrum_a_fcn} and Thm.~\ref{thm:spectrum_regular} whenever \cite[Thm.~B]{Lalley} is used, we gain \cite[Prop.~7{.}3, 7{.}4]{Lalley} when $-\delta<t^*$. 
  For this we need the follwing notions: A function $\xi\in\mathcal C(E^{\infty},\mathbb R)$ is called \emph{lattice} if there exists $\psi\in\mathcal C(E^{\infty},\mathbb R)$ such that the range of $\xi-\psi+\psi\circ\sigma$ lies in a discrete subgroup of $\mathbb R$. Otherwise, we say that $\xi$ is \emph{non-lattice}.
  \begin{prop}[cf.\ {\cite[Prop.~7{.}3, 7{.}4]{Lalley}}]\label{prop:holomorphic}
   Suppose that $-\delta<t^*$.
   \begin{enumerate}[label={{\rm (\roman*)}}]
    \item If $\xi$ is non-lattice, then $z\mapsto(\Id -\mathcal L_{\eta+z\xi})^{-1}$ is holomorphic in a neighbourhood of every $z$ on the line $\Re(z)=-\delta$ except for $z=-\delta$.
    \item If $\xi$ is integer-valued but there does not exist any $\psi\in\mathcal C(E^{\infty})$ such that the range of $\xi-\psi+\psi\circ\sigma$ is contained in a proper subgroup of $\mathbb Z$, then $z\mapsto(\Id -\mathcal L_{\eta+z\xi})^{-1}$ is $2\pi\mathbf{i}$-periodic, and holomorphic at every $z$ on the line $\Re(z)=-\delta$ such that $\Im(z)/(2\pi)$ is not an integer.
   \end{enumerate}
  \end{prop}

\section{Renewal theory}\label{sec:renewal}

In \cite{Lalley} renewal theorems for counting measures in symbolic dynamics were established, where the underlying symbolic space is based on a finite alphabet.  {(Given a measurable space $(\Omega, \mathcal A)$, a \emph{counting measure} $\mu_A$ on $A\in\mathcal A$ is defined through $\mu_A(B)=\# A\cap B$ for $B\in\mathcal A$.)} These renewal theorems were extended to more general measures in \cite{Diss,renewal}.
The renewal theorems of \cite{renewal} generalise and unify
\begin{inparaenum}
	\item \cite[Thms.~1 and 2]{Lalley}
	\item the classical key renewal theorem for finitely supported probability measures \cite{Feller2} and
	\item a class of Markov renewal theorems (see e.\,g. \cite{Alsmeyer,Asmussen}).
\end{inparaenum} 
In the present section we show how the results of Sec.~\ref{sec:MU} and \ref{sec:spectral} lead to an extension of the generalised versions from \cite{renewal} to the setting of an underlying countable alphabet. Moreover, we explain that the classical key renewal theorem for arbitrary discrete measures \cite{Feller2} is a simple special case of the new renewal theorem.
Having the results of Sec.~\ref{sec:symbolicrenewal}, the proof of the newly extended renewal theorem follows along the lines of proof in \cite{renewal,Lalley}. Therefore, we only present an outline of proof focusing on the necessary modifications, in Sec.~\ref{sec:proofoutline}.

We fix $\theta\in(0,1)$ and a non-negative but not identically zero $\chi\in\mathcal F^b_{\theta}(E^{\infty},\mathbb R)$. Let $\xi,\eta\in\mathcal F_{\theta}(E^{\infty},\mathbb R)$ satisfy the following: 
\begin{enumerate}[label=(\Alph*)]
 \item \label{it:regularPotential}\emph{Regular potential}.
  $\xi$ is non-negative, but not identically zero. There exists a unique $\delta\in\mathbb R$ with $P(\eta-\delta\xi)=0$. Further, $\int -(\eta +t\xi)\,\textup{d}\mu_{\eta-\delta\xi}<\infty$ for all $t$ in a neighbourhood of $-\delta$ and $-\delta<t^*\defeq\sup\{t\in\mathbb R\mid \eta+t\xi\ \text{is summable}\}$.
\end{enumerate}
For $x\in E^{\infty}$ we study the asymptotic behaviour as $t\to\infty$ of the renewal function
\begin{equation}
 N(t,x)\defeq\sum_{n=0}^{\infty}\sum_{y:\sigma^ny=x}\chi(y)f_y(t-S_n\xi(y))\textup{e}^{S_n\eta(y)},
\end{equation}
where $f_x\colon\mathbb R\to\mathbb R$, for $x\in E^{\infty}$, needs to satisfy some regularity conditions (see \ref{it:Lebesgue}--\ref{it:decay} below). We call $N$ a \emph{renewal function} since it satisfies an analogue to the classical renewal equation:
\begin{equation}\label{eq:renewaleq}
 N(t,x)\defeq\sum_{y:\sigma y=x} N(t-\xi(y),y)\textup{e}^{\eta(y)}+\chi(x)f_x(t).
\end{equation}

\begin{enumerate}[label=(\Alph*)]\setcounter{enumi}{1}
 \item\label{it:Lebesgue} \emph{Lebesgue integrability}. For any $x\in E^{\infty}$ the  Lebesgue integral 
 \[
  \int_{-\infty}^{\infty}\textup{e}^{-t\delta}\lvert f_x(t)\rvert\textup{d}t
 \]
 exists.
 \item \label{it:boundedC}\emph{Boundedness of $N$}. There exists $c>0$ such that $\textup{e}^{-t\delta}N^{\text{abs}}(t,x)\leq c$ for all $x\in E^{\infty}$ and $t\in\mathbb R$, where
 \begin{equation*}
  N^{\text{abs}}(t,x)\defeq\sum_{n=0}^{\infty}\sum_{y:\sigma^ny=x}\chi(y)\lvert f_y(t-S_n\xi(y))\rvert\textup{e}^{S_n\eta(y)}.
 \end{equation*}
 \item\label{it:decay} \emph{Exponential decay of $N$ on the negative half-axis}. There exist $\widetilde{c}>0,s>0$ and $t_0\in\mathbb R$ such that $\textup{e}^{-t\delta}N^{\text{abs}}(t,x)\leq \widetilde{c}\textup{e}^{st}$ for all $t\leq t_0$.
\end{enumerate}
The asymptotic behaviour of $N$ as $t\to\infty$ depends on whether the potential $\xi$ is lattice or non-lattice. 
Two functions $f,g\colon\mathbb R\to\mathbb R$ are called \emph{asymptotic} as $t\to\infty$, written $f(t)\sim g(t)$ as $t\to\infty$, if for all $\e>0$ there exists $\widetilde{t}\in\mathbb R$ such that for all $t\geq \widetilde{t}$ the value $f(t)$ lies between $(1-\e)g(t)$ and $(1+\e)g(t)$. 
For $t\in\mathbb R$ we define $\lfloor t\rfloor\defeq\max\{k\in\mathbb Z\mid k\leq t\}$ and $\{t\}\defeq t-\lfloor t\rfloor\in[0,1)$. Note that $\lfloor t\rfloor$ and $\{t\}$ respectively are the integer and the fractional part of $t$ if $t\geq 0$.

\begin{thm}[Renewal theorem]\label{thm:RT1}
   Assume that  $x\mapsto f_x(t)$ is $\theta$-H\"older continuous for every $t\in\mathbb R$ and that Conditions \ref{it:regularPotential} to \ref{it:decay} hold.
   \begin{enumerate}[label={{\rm (\roman*)}}]
   \item\label{it:RT1:nl} If $\xi$ is non-lattice and $f_x$ is monotonic for every $x\in E^{\infty}$, then 
   \begin{equation*}
    N(t,x)\sim\textup{e}^{t\delta}\eigenf_{\eta-\delta \xi}(x)\underbrace{
\frac{1}{\int \xi\textup{d}\mu_{\eta-\delta \xi}}\int_{E^{\infty}} \chi(y)\int_{-\infty}^{\infty}\textup{e}^{-T\delta}f_y(T)\textup{d}T\textup{d}\nu_{\eta-\delta \xi}(y)
}_{\eqdef G}
   \end{equation*}
   as $t\to\infty$, uniformly for $x\in E^{\infty}$.
   \item\label{it:RT1:l} Assume that $\xi$ is lattice and let $\zeta,\psi\in\mathcal{C}(E^{\infty},\mathbb R)$ satisfy the relation
   \[
   \xi-\zeta=\psi-\psi\circ\sigma,
   \]
   where $\zeta$ is a function whose range is contained in a discrete subgroup of $\mathbb R$. Let $a>0$ be maximal such that $\zeta(E^{\infty})\subseteq a\mathbb Z$. Then
  \begin{align*}
    N(t,x)
    \sim \textup{e}^{t\delta}\eigenf_{\eta-\delta\zeta}(x)\widetilde{G}_x(t)
  \end{align*}
  as $t\to\infty$, uniformly for $x\in E^{\infty}$, where $\widetilde{G}_x$ is periodic with period $a$ and
  \begin{align*}
    \widetilde{G}_x(t)
    &\defeq \int_{E^{\infty}}\chi(y)
\sum_{\ell=-\infty}^{\infty}\textup{e}^{-a \ell\delta}f_y\left(a \ell+a\left\{\tfrac{t+\psi(x)}{a}\right\}-\psi(y)\right)
    \textup{d}\nu_{\eta-\delta\zeta}(y)\\
    &\qquad\times \textup{e}^{-a\big{\{}\frac{t+\psi(x)}{a}\big{\}}\delta}
    \frac{a\textup{e}^{\delta\psi(x)}}{\int\zeta\textup{d}\mu_{\eta-\delta\zeta}}.
  \end{align*}
  \item\label{it:RT1:av}We always have 
  \begin{equation*}
  \lim_{t\to\infty}\frac{1}{t}\int_0^{t}\textup{e}^{T\delta}N(T,x)\textup{d}T=G\cdot\eigenf_{\eta-\delta \xi}(x).
  \end{equation*}
  \end{enumerate}
\end{thm}

\begin{rem}\label{rem:monotonicity}
 The monotonicity condition in Thm.~\ref{thm:RT1} \ref{it:RT1:nl} can be substituted by other conditions (see \cite{renewal}). One such condition is that there exists $n\in\mathbb N$ for which $S_n\xi$ is bounded away from zero and that the family $(t\mapsto\textup{e}^{-t\delta}\lvert f_x(t)\rvert\mid x\in E^{\infty})$ is equi directly Riemann integrable; a condition which is motivated by the key renewal theorem, see \cite{Feller2}.
\end{rem}

Our renewal theorem deals with renewal functions which act on the product space $\mathbb R\times E^{\infty}$. When making the restrictions  {\ref{it:Idiscrete} to \ref{it:constpot} below} the renewal function is independent of the second component and we obtain the classical key renewal theorem for discrete measures (see \cite{renewal} for details). 
\begin{enumerate}
	\item\label{it:Idiscrete}  $I=\mathbb N$ or $I=\{1,\ldots,M\}$ for some $M\in\mathbb N$ and $E^{\infty}= I^{\mathbb N}$ (full shift).
	\item $\renfcn_x = \renfcn$ is independent of $x\in I^{\mathbb N}$
	\item $\codefunc=\mathds 1$
	\item\label{it:constpot} $\codefun$ and $\eta$ are constant on cylinder sets of length one. 
\end{enumerate}
Notice, any probability vector $(p_1,p_2,\ldots)$ with $p_i\in(0,1)$ and every $(s_1,s_2,\ldots)$ with $s_i\geq 0$  determine $\eta,\xi\in\mathcal F_{\theta}(E^{\infty},\mathbb R)$ via  $\eta(\om_1\om_2\cdots)\coloneqq \log(p_{\om_1}\ee^{\delta s_{\om_1}})$ and $\xi(\om_1\om_2\cdots)\coloneqq s_{\om_1}$.
Setting $Z(t)\defeq\ee^{-\mdim t}N(t,x)$, which now is independent of $x$, and $z(t)\defeq\ee^{-\mdim t}\renfcn(t)$ (which is directly Riemann integrable  by Rem.~\ref{rem:monotonicity}) we deduce from \eqref{eq:renewaleq} that $Z$ solves the \emph{classical renewal equation}:
\begin{align}\label{eq:reneqFeller}
        Z(t)
        =\sum_{i\in I} Z(t - s_i)p_i+ z(t)
\end{align}
for $t\in\mathbb R$ or equivalently, $Z=Z\star F+z$, where $F$ is the distribution which assigns mass $p_i$ to $s_i$ and where $\star$ denotes the convolution operator.
Thm.~\ref{thm:RT1} implies
\begin{cor}[Key renewal theorem for discrete measures, see e.\,g.\,{\cite[Ch.~XI]{Feller2}}]\label{cor:keyrt}
  Let $s_1,s_2,\ldots\geq 0$ be so that there exists an $n\in \mathbb N$ with $s_n>0$ and let $(p_1,p_2\ldots)$ be a probability vector with $p_i\in(0,1)$. Denote by $z\colon\mathbb R\to\mathbb R$ a directly Riemann integrable function with $z(t)\leq c'\ee^{st}$ for all $t<0$ and some $c',s>0$. Further, let $Z\colon\mathbb R\to\mathbb R$ be the unique solution of the renewal equation \eqref{eq:reneqFeller} which satisfies $\lim_{t\to -\infty} Z(t)=0$. 
Then the following hold:
\begin{enumerate}[label={{\rm (\roman*)}}]
        \item If  $\{s_1,s_2,\ldots\}$ is not contained in a discrete subgroup of $\mathbb R$, then as $t\to \infty$
        \[
        Z(t)\sim \frac{1}{\sum_{i} p_i s_i}\int_{-\infty}^{\infty} z(T)\textup{d}T.
        \]
        \item If $\{s_1,s_2,\ldots\}\subset \aaa\cdot\mathbb Z$ and $\aaa>0$ is maximal, then as $t\to \infty$
        \[
        Z(t)\sim \frac{\aaa}{\sum_{i} p_i s_i}\sum_{\ell=-\infty}^{\infty} z(\aaa\ell+t).
        \]
        \item We always have 
        \[
        \lim_{t\to\infty}t^{-1}\int_0^tZ(T)\textup{d}T
        = \frac{1}{\sum_{i} p_i s_i}\int_{-\infty}^{\infty} z(T)\textup{d}T.
        \]      
\end{enumerate}
\end{cor}

\subsection{The ideas of proof}\label{sec:proofoutline}
In proving Thm.~\ref{thm:RT1} we use the methods of proof which were developed in \cite{Lalley} and extended in \cite{renewal}. Besides using our new results of Sec.~\ref{sec:symbolicrenewal} only small modifications are necessary. Thus, below, we only provide an outline of the main steps of the proof and focus on the necessary modifications. For more details we refer the reader to \cite{renewal}, where similar notation is used.

\begin{proof}[Outline of the proof of Thm.~\ref{thm:RT1}\ref{it:RT1:nl}] 
For $z\in\mathbb C$ and $x\in E^{\infty}$ one studies the Fourier-Laplace transform
\begin{equation}
  \Laplace(z,x)\defeq\int_{-\infty}^{\infty}\ee^{zT}\ee^{-T\mdim}N(T,x)\textup{d}T
\end{equation} 
of $t\mapsto\ee^{-t\mdim}N(t,x)$ at $z$.
Conditions \ref{it:boundedC}, \ref{it:decay} and the monotone and dominated convergence theorems imply that for any sufficiently small fixed $\e>0$ and all $z\in\mathbb C$ with $-s+\e\leq\Re(z)\leq-\e$ one has 
\begin{align*}
  \Laplace(z,x)
  &=\sum_{n=0}^{\infty}\PF_{\eta+(z-\mdim)\codefun}^n\left(\codefunc\int_{-\infty}^{\infty}\ee^{(z-\mdim)T}\renfcn_{\cdot}(T)\textup{d}T\right)(x),
\end{align*}
where $\renfcn_{\cdot}(T)\colon E^{\infty}\to\mathbb R$, $x\mapsto \renfcn_x(T)$.

In the present setting, we assume that $\mdim$ is unique with $P(\eta-\mdim\xi)=0$. Together with the monotonicity of $t\mapsto P(\eta+t\xi)$ this implies  $\gamma_{\eta+t\xi}<1$ for $t<-\delta$. Applying Prop.~\ref{prop:spectrum_a_fcn}, Thm.~\ref{thm:spectrum_regular} and the spectral radius formula shows that the above series converges for $-\alpha<\Re(z)<0$ for some $\alpha\in(0,s]$, whence
\[
  \Laplace(z,x)
  =(\Id -\PF_{\eta+(z-\mdim)\codefun})^{-1}\left(\codefunc\int_{-\infty}^{\infty}\ee^{(z-\mdim)T}\renfcn_{\cdot}(T)\textup{d}T\right)(x).
\]
Using Prop.~\ref{prop:Lalley:7.2} and Cor.~\ref{cor:residue}, we see that $z\mapsto\Laplace(z,x)$ has a simple pole at $z=0$ with  residue
\begin{equation}
  \frac{\int_{E^{\infty}}\codefunc(y)\int_{-\infty}^{\infty}\ee^{-T\mdim}\renfcn_y(T)\textup{d}T\textup{d}\nu_{\eta-\mdim\codefun}(y)}{\int\codefun\textup{d}\mu_{\eta-\mdim\codefun}}\eigenf_{\eta-\mdim\codefun}(x)\eqqcolon-U(x),
\end{equation}
where $U(x)=G\cdot\eigenf_{\eta-\delta \xi}(x)$ and $G$ is as in Thm.~\ref{thm:RT1}\ref{it:RT1:nl}.
Thus, $\Laplace(z,x)$ has the following representation.
\begin{equation}\label{eq:Laplaceq}
  \Laplace(z,x)=q(z,x)-\frac{U(x)}{z},
\end{equation}
where $q(\cdot,x)\colon\mathbb C\to\mathbb C$, $z\mapsto q(z,x)$ is holomorphic in a region containing the strip $\{z\in\mathbb C\mid -\alpha+\e\leq\Re(z)\leq 0\}$ with sufficiently small $\e>0$. 
A test function argument yields that it suffices to show that 
\begin{align}
  \begin{aligned}\label{eq:parseval}
    &\lim_{\beta\searrow 0}\int_{-\infty}^{\infty}\Laplace(\im\theta-\beta,x)\hat{\prob_{\e}}(\im\theta)\ee^{-\im\theta r}\frac{\textup{d}\theta}{2\pi}\\
    &\ \stackrel{(\ref{eq:Laplaceq})}{=} \lim_{\beta\searrow 0}\int_{-\infty}^{\infty}\left(q(\im\theta-\beta,x)+\frac{U(x)(\im\theta+\beta)}{\theta^2+\beta^2}\right)
    \hat{\prob_{\e}}(\im\theta)\ee^{-\im\theta r}\frac{\textup{d}\theta}{2\pi}
  \end{aligned}
\end{align}
converges to $U(x)$ as $r\to\infty$. Here, $\hat{\prob_{\e}}(\im\theta)=\hat{\prob}(\im\theta\e/\tau(\e))$ with 
\begin{equation*}
 \widehat{\Pi}(\mathbf i\theta)\defeq
 \begin{cases}
  \exp\left(\tfrac{-\theta^2}{1-\theta^2}\right) &\colon |\theta|\leq 1\\
  0 &\colon\text{otherwise}
 \end{cases}
\end{equation*}
and a particular decreasing function $\tau$, which satisfies $\lim_{\e\searrow 0}\tau(\e)=\infty$.
The convergence of \eqref{eq:parseval} to $U(x)=G\cdot\eigenf_{\eta-\delta \xi}(x)$ as $r\to\infty$ is shown by means of complex analysis (see  \cite[Sec.~5{.}2]{renewal}, \cite{Lalley}).
\end{proof}

\begin{proof}[Outline of the proof of Thm.~\ref{thm:RT1}{\rm\ref{it:RT1:l}}] 
In the lattice situation we work with discrete Fourier-Laplace transforms. Conditions \ref{it:boundedC}, \ref{it:decay} imply that for fixed $\beta\in[0,\aaa)$ and $x\in E^{\infty}$, the function $\hat{N}^{\beta}(\cdot,x)$ given by 
\begin{equation}
	\hat{N}^{\beta}(z,x)\defeq\sum_{\ell=-\infty}^{\infty}\ee^{\ell z}N(\aaa l+\beta-\psi(x),x)
\end{equation}
is well-defined and analytic on $\{z\in\mathbb C\mid -\aaa(s+\mdim)<\Re(z)<-\aaa\mdim\}$.
Using $S_n\xi=S_n\z+\psi-\psi\circ\sigma^n$ and $S_n\z\in\aaa\mathbb Z$ for all $n\in\mathbb N$, Conditions \ref{it:boundedC}, \ref{it:decay} imply for such $z$ that 
\begin{align*}
	\hat{N}^{\beta}(z,x)
        &=\sum_{n=0}^{\infty}\PF^n_{\eta+\aaa^{-1}z\z}\left(\codefunc
	\sum_{\ell=-\infty}^{\infty}\ee^{\ell z}\renfcn_{\cdot}(\aaa l+\beta-\psi)\right)(x)
\end{align*}
with $\renfcn_{\cdot}(t)\colon E^{\infty}\to\mathbb R$, $x\mapsto \renfcn_x(t)$ as before.
With the same arguments as in the proof of the non-lattice situation, there exists $\alpha\in(0,s]$ so that $\gamma_{\eta+\aaa^{-1}z\z}=\gamma_{\eta+\aaa^{-1}z\xi}<1$ if $z\in\mathcal Z$, where
\[
 \mathcal Z\defeq \{z\in\mathbb C\mid -\aaa(\alpha+\delta)<Re(z)<-\aaa\delta\}.
\]
The spectral radius formula now implies, for $z\in\mathcal Z$, that
\[
	\hat{N}^{\beta}(z,x)
	=(\Id -\PF_{\eta+\aaa^{-1}z\z})^{-1}\left(
\codefunc\sum_{\ell=-\infty}^{\infty} \ee^{\ell z}\renfcn_{\cdot}(\aaa\ell+\beta-\psi)\right)(x).
\]
Note that $\|\codefunc\sum_{\ell=-\infty}^{\infty} \ee^{\ell z}f_{\cdot}(\aaa\ell+\beta-\psi)\|_{\infty}$ is finite because of Conditions \ref{it:boundedC}, \ref{it:decay}.

Since $\aaa^{-1}\z$ is integer-valued but not co-homologous to any function valued in a proper subgroup of the integers, we can apply Prop.~\ref{prop:holomorphic}. Therefore, $z\mapsto (\Id -\PF_{\eta+\aaa^{-1}z\z})^{-1}$
is holomorphic at each $z=-\aaa\mdim+\im\theta$, for $0<\lvert\theta\rvert\leq\pi$. 
Moreover, it has a simple pole at $z=-\aaa\delta$ with residue
\begin{align*}
   C_{\beta}(x)
   \defeq-\frac{\aaa}{\int\z\textup{d}\mu_{\eta-\mdim\z}}
   \int\codefunc(y)\sum_{\ell=-\infty}^{\infty}\ee^{-\aaa\ell \delta}\renfcn_y(\aaa\ell+\beta-\psi(y))\textup{d}\nu_{\eta-\delta\z}(y) \eigenf_{\eta-\delta\z}(x),
\end{align*}
see Cor.~\ref{cor:residue}.
It follows that $\hat{N}^{\beta}(\cdot,x)\colon\mathbb C\to\mathbb C$, $z\mapsto \hat{N}^{\beta}(z,x)$ is meromorphic in
\[
        \tilde{\mathcal Z}(\e)\defeq\{z\in\mathbb C\mid-\aaa(\mdim+\alpha)<\Re(z)<-\aaa\mdim+\e,\ 0\leq\Im(z)\leq\pi\},
\]
for some $\e>0$, and that the only singularity in this region is a simple pole at $-\aaa\mdim$ with residue $C_{\beta}(x)$. Thus,
\[
	\sum_{\ell=0}^{\infty}\ee^{\ell z}N(\aaa\ell+\beta-\psi(x),x)-\frac{C_{\beta}(x)}{z+\aaa\mdim}
\]
is holomorphic in $\tilde{\mathcal{Z}}(\e)$, whence 
\[
	\sum_{\ell=0}^{\infty}z^{\ell}\ee^{-\aaa\ell\mdim}N(\aaa\ell+\beta-\psi(x),x)-\frac{C_{\beta}(x)}{z-1}
\]
is holomorphic in $\{\ee^{z+\aaa\mdim}\mid z\in\tilde{\mathcal{Z}}(\e)\}$. This implies that 
\[
	L(z,x)\defeq\sum_{\ell=0}^{\infty}z^{\ell}\left(\ee^{-\aaa\ell\mdim}N(\aaa\ell+\beta-\psi(x),x)+C_{\beta}(x)\right)
\]
is holomorphic in $\{z\mid\lvert z\rvert<\ee^{\e}\}$. 
Since $\ee^{\e}>1$, the coefficient sequence of the power series of $L(\cdot,x)\colon\mathbb C\to\mathbb C$, $z\mapsto L(z,x)$ converges to zero exponentially fast, more precisely, 
\[
	\ee^{-\aaa n\mdim}N(\aaa n+\beta-\psi(x),x)+C_{\beta}(x)
	\in\mathfrak{o}((1+(\ee^{\e}-1)/2)^{-n})
\]
as $n\to\infty$ ($n\in\mathbb N$).
Thus, for $x\in E^{\infty}$ we have
\begin{align*}
	N(t,x)
	&=N\bigg(\aaa\underbrace{\left\lfloor \frac{t+\psi(x)}{\aaa}\right\rfloor}_{\eqdef n}+\underbrace{\aaa\Big{\{}\frac{t+\psi(x)}{\aaa}\Big{\}}}_{\eqdef\beta}-\psi(x),x\bigg)\\
	&\sim-\ee^{\aaa\left\lfloor\frac{t+\psi(x)}{\aaa}\right\rfloor\mdim}C_{\aaa\{(t+\psi(x))/{\aaa}\}}(x)\\
        &=\ee^{t\mdim}\ee^{-\aaa\big{\{}\frac{t+\psi(x)}{\aaa}\big{\}}\mdim}\ee^{\mdim\psi(x)}\frac{\aaa}{\int\z\textup{d}\mu_{\eta-\mdim\z}}\eigenf_{\eta-\mdim\z}(x)\\
        &\qquad\times\int_{E^{\infty}}\codefunc(y)\sum_{\ell=-\infty}^{\infty}\ee^{-\ell\aaa\mdim}\renfcn_y\left(\aaa\ell+\aaa\Big{\{}\frac{t+\psi(x)}{\aaa}\Big{\}}-\psi(y)\right)\textup{d}\nu_{\eta-\mdim\z}(y)\\
        &= \ee^{t\mdim}\eigenf_{\eta-\mdim\z}(x)\tilde{G}_x(t)
\end{align*}
as $t\to\infty$. Since in all instances where $t$ occurs only the fractional part is involved, it is clear that $\tilde{G}_x$ is periodic with period $\aaa$, which finishes the proof.
\end{proof}

\bibliographystyle{alpha}
\bibliography{Apollonian}

\end{document}